\documentclass[12pt, reqno]{amsart}

\usepackage{amssymb,latexsym,amsmath,amsfonts}
\usepackage{enumitem}
\usepackage{mathrsfs}
\usepackage{graphicx}
\usepackage{hyperref}
\usepackage[usenames]{color}
\usepackage{comment}

\DeclareFontFamily{OT1}{rsfs}{}
\DeclareFontShape{OT1}{rsfs}{n}{it}{<-> rsfs10}{}
\DeclareMathAlphabet{\mathscr}{OT1}{rsfs}{n}{it}

\numberwithin{equation}{section}

\theoremstyle{definition}
\newtheorem{definition}{Definition}[section]

\theoremstyle{remark}
\newtheorem{remark}[definition]{Remark}

\theoremstyle{plain}
\newtheorem{theorem}[definition]{Theorem}
\newtheorem{result}[definition]{Result}
\newtheorem{lemma}[definition]{Lemma}

\newtheorem{corollary}[definition]{Corollary}

\definecolor{DPurple}{rgb}{0.76,0.2,0.69}

\usepackage{color}

\newcommand{\nice}{\Sigma}
\newcommand{\D}{\mathbb D}

\newcommand{\capa}{{c}}
\newcommand{\gr}{{\rm{g}}}
\newcommand{\Gal}{{\rm {Conj}}}
\newcommand{\esc}{{\mathcal E}(\nice; \{\alpha_n\})}
\newcommand{\on}{\gamma}
\newcommand{\away}{\zeta}

\newcommand{\sph}{\widehat{\mathbb{C}}}
\newcommand{\C}{\mathbb{C}} 
\newcommand{\R}{\mathbb{R}}
\newcommand{\Z}{\mathbb{Z}}
\newcommand{\N}{\mathbb{N}}
\newcommand{\Q}{\mathbb{Q}}

\begin{document}

\title[Around Fekete's theorem]{Around Fekete's theorem}

\author{Norm Levenberg \and Mayuresh Londhe}
\thanks{Norm Levenberg is supported by Simons Foundation grant No. 707450}
\address{Department of Mathematics, Indiana University, Bloomington, Indiana 47405, USA}
\email{nlevenbe@iu.edu, mmlondhe@iu.edu}

\begin{abstract}
A classical result of Fekete gives necessary conditions on a compact set in the complex plane so that
it contains infinitely many sets of conjugate algebraic integers.
For such sets, we demonstrate the existence of a sequence of algebraic integers such that most of 
their conjugates eventually lie near the set, while maintaining a bound on heights. 
Finally, we examine properties satisfied by the limiting distribution of a sequence of algebraic numbers.
\end{abstract}


\maketitle

\section{Introduction}\label{S:intro}

Let $\nice$ be a compact set in the complex plane $\C$.
Fekete studied which compact sets $\nice$ contain
only a finite number of sets of conjugate algebraic integers.
This involves the notion of the {\textit{transfinite diameter of}} $\nice$;
it is given by
\[
d(\nice)= \lim_{n \to \infty} \Big(  \sup_{\xi_1, \dots, \xi_n \in \nice} \prod_{i<j} \vert \xi_i- \xi_j \vert^{\frac{2}{n(n-1)}} \Big).
\]
The transfinite diameter of a circle of radius $r$ is equal to $r$, and for 
a line segment of length $l$ it is equal to $l/4$. It turns out that the quantity
$d(\nice)$ coincides with the logarithmic capacity of $\nice$, denoted by $\capa(\nice)$; see 
Subsection~\ref{SS:potential_theory} for the definition. The product $\prod_{i<j} \vert \xi_i- \xi_j|$ 
is related to the resultant of $P$ and its derivative $P'$, where $P(z)=\prod_{j=1}^n (z-\xi_j)$; see (\ref{disc}). 
This motivates the use of potential-theoretic tools in studying sets of conjugate algebraic numbers.
We now state a result by Fekete: 

\begin{result}[Fekete, \cite{Fekete:udvwbgmg23}]\label{R:finite_integers}
Let $\nice \subset \C$ be compact and symmetric about the real axis. 
If $\capa(\nice) <1$, then there exists a neighborhood $U$ of $\nice$
such that there are only a finite number of 
algebraic integers $\alpha$ so that $\alpha$ and all its conjugates lie in $U$.
\end{result}

{\noindent The above result is not true if we replace algebraic integers 
by algebraic numbers in the statement,
for example, if $\nice=\{|z|=1/2\}$ then all roots of $2^nz^n-1$ are in $\nice$.
In \cite{FekeSzego:aeicwr55}, Fekete, together with Szeg\"{o} showed that the converse 
of Result~\ref{R:finite_integers} is also true.%
\smallskip

Our main results involve a notion of height 
$h_{\nice} ({\alpha})$ 
of an algebraic number $\alpha$ relative 
to a compact set $\nice$ as defined in \cite{rumely:bilu99}.
Let $\overline \Q$ denote the set of algebraic numbers in $\C$. 
Given a nonzero $\alpha \in \overline \Q$, there is a polynomial $P_{\alpha}$
with integer coefficients and leading coefficient $b_{\alpha}$ and degree $\deg \alpha$ 
that is irreducible over the integers and satisfies $P_{\alpha}(\alpha)=0$
($P_{\alpha}$ is unique up to the sign of $b_{\alpha}$). For an algebraic integer $\alpha$, $b_{\alpha}=1$ or $-1$. 
We denote by $\Gal (\alpha)$ the set of all conjugates of $\alpha$ which is same as the set
of all roots of $P_{\alpha}$ in $\C$.
Recall that the (absolute logarithmic) Weil height of $\alpha \in \overline \Q$ is given by
\[
h({\alpha}) = \frac{1}{\deg \alpha} \Big(\log |b_{\alpha}|+\sum_{x \in \Gal (\alpha)}\max \{0, \log|x|\}\Big).
\]
We refer the reader to \cite{BomGub:hidg06} for an introduction to the theory of heights.
Note that $\log^+|z|:= \max \{0, \log|z|\}$ is the Green's function of the unit circle centered at 0
with pole at $\infty$.
For an arbitary compact set $\nice$ with positive capacity,
the {\textit{height of $\alpha \in \overline \Q$ with respect to}} $\nice$ is defined as
\begin{equation}\label{height}
h_{\nice} ({\alpha}) =  \frac{1}{\deg\alpha} \Big(\log |b_{\alpha}|+\sum_{x \in \Gal (\alpha)}\gr_{\nice} (x, \infty)\Big),
\end{equation}
where $\gr_{\nice} (z, \infty)$ denotes the Green's function of $\nice$ with pole at $\infty$.
Loosely speaking, $h_{\nice} ({\alpha})$ quantifies the distances of conjugates of $\alpha$ from $\nice$. 
See Section~\ref{S:prelim} for more details about the Green's function and height $h_{\nice}$.
Also, see \cite{rumely:bilu99} and \cite{Pritsker:crelle11} for an application of this height
to limit distributions of sets of conjugate algebraic numbers.%
\smallskip

In the case $\capa(\nice) \geq 1$ and $\nice$ is symmetric about the real axis,
it follows from the Fekete-Szeg\"o result in \cite{FekeSzego:aeicwr55} that
there exists a sequence $\{\alpha_n\}$ in $\overline \Q$ satisfying
$h_{\nice} (\alpha_n) \to 0$. 
\begin{result}\label{R:lower}
If $0 <\capa(\nice) < 1$, 
then for every sequence $\{\alpha_n\}$ in $\overline \Q$
such that $\deg \alpha_n \to \infty$, we have
\begin{equation}\label{E:lower}
\liminf_{n \to \infty} h_{\nice} (\alpha_n) \geq \frac{1}{2} \log \frac{1}{\capa(\nice)}.
\end{equation}
\end{result}
\noindent This follows for regular $\nice$ from \cite[Th\'eor\`eme~1]{FRL:eqpphd06}. We also give an independent 
proof of \eqref{E:lower} for general sets using Theorem~\ref{T:limiting} below.
The Northcott property of $h_{\nice}$ gives us
the following formulation which can be seen as an effective version of Result~\ref{R:finite_integers}:}
for any $\epsilon >0$
there are only a finite number of algebraic {\textbf{numbers}} $\alpha$ such that
\[
h_{\nice} (\alpha) \leq \frac{1}{2} \log \frac{1}{\capa(\nice)} - \epsilon.
\]
However, it is not clear if there
exists a sequence $\{\alpha_n\}$ in $\overline \Q$ satisfying
$h_{\nice} (\alpha_n) \to  -(1/2) \log {\capa(\nice)}$. We show that there exists a sequence $\{\alpha_n\}$ of algebraic integers such that
$h_{\nice} (\alpha_n) \to  - \log {\capa(\nice)}$, and, in addition, 
``most'' conjugates of $\alpha_n$ are eventually near $\nice$.
This latter property can be achieved without the condition on height, 
for instance, utilizing a result of Motzkin \cite{Mo47}.

\begin{theorem}\label{T:existence}
Let $\nice \subset \C$ be compact and symmetric about the real axis. If $0< \capa(\nice) <1$,
then there exists a sequence $\{\alpha_n\}$ of 
distinct algebraic integers such that
$h_{\nice}(\alpha_n) \to - \log {\capa(\nice)}$ as $n \to \infty$,
and for any neighborhood $U$ of $\nice$, we have
\begin{equation}\label{E:most}
\frac{\sharp \{ \Gal (\alpha_n) \cap U\}}{\deg \alpha_n} \longrightarrow 1 
\quad {\text as } \quad n \longrightarrow \infty,
\end{equation}
where $\sharp \{\Gal (\alpha_n) \cap U\}$ denotes the number of conjugates 
of $\alpha_n$ contained in $U$.
\end{theorem}

Owing to Result~\ref{R:finite_integers}, $\deg{\alpha_n} \to \infty$ in the above theorem.
We briefly outline the idea of the proof of  Theorem~\ref{T:existence}. 
We construct a sequence of auxiliary sets $E_n$ so that
$\capa(\nice \cup E_n)=1$. The set $E_n$ is chosen to be a subset of the level set
$\gr_{\nice} (z, \infty)=n$ and symmetric about the real axis. Let $\nu_n$ be the equilibrium measure 
of $\nice \cup E_n$. 
We solve a Dirichlet problem to determine how much mass $\nu_n$ puts on the set
$E_n$. Using this information in the definition of height together with the Fekete-Szeg\"o result 
in \cite{FekeSzego:aeicwr55}, we construct the appropriate sequences $\{\alpha_n\}$.
\smallskip

In fact, we can find a sequence $\{\alpha_n\}$ as in Theorem~\ref{T:existence} such that
the limiting distribution of the probability measures uniformly distributed on all conjugates of $\alpha_n$ 
is given by the equilibrium measure of $\nice$ (see Subsection~\ref{SS:potential_theory} for
the definition; also see Remark~\ref{rmk}). However, for a sequence $\{\alpha_n\}$ satisfying the hypothesis
of Theorem~\ref{T:existence}, we do not yet know whether the limiting distribution is given by the 
equilibrium measure of $\nice$. Therefore, we now explore the properties satisfied by the limiting measures.%
\smallskip

Let $\nice \subset \C$ be compact and symmetric about the real axis, and assume that
$\nice$ is non-polar.
Let $\{\alpha_n\}$ be a sequence
in $\overline \Q$ for which $\{h_{\nice} (\alpha_n)\}$ is bounded. 
Let $\mu_n$ be the probability measure uniformly distributed over all conjugates of $\alpha_n$, i.e.,
\[
\mu_n:= \frac{1}{\deg \alpha_n} \sum_{x \in \Gal (\alpha_n)} \delta_x.
\]
It follows from Lemma \ref{L:tight} that the sequence $\{\mu_n\}$ has a convergent subsequence. Although a limiting measure 
$\mu$ need not be compactly supported, it turns out that such a 
$\mu$ has finite logarithmic energy $I(\mu)$ and satisfies
$\int \log^+ |z|\, d\mu < \infty.$
This allows us to use potential-theoretic methods.
After passing to a subsequence, we can assume that the sequences $\{h_{\nice} (\alpha_n)\}$
and $\{{\log |b_{n}|}/{\deg \alpha_{n}}\}$ are convergent,
where $b_{n}:= b_{\alpha_n}$.
For such a sequence,
 $\esc$ is the escape rate as defined in (\ref{E:escape_rate}). 
Refining the argument in Serre \cite{Serre}, we show that
$\mu$ satisfies the following arithmetic property: 

\begin{theorem}\label{T:limiting}
For any polynomial $Q \in \Z[z]$ with integer coefficients, we have
\[
\frac{1}{\deg Q} \int \log \vert Q(z) \vert \,d\mu(z) \geq -  \lim_{n \to \infty} \frac{\log |b_n|}{\deg \alpha_n} - \esc.
\]
Furthermore, we have the following bound on the logarithmic energy of $\mu$:
\[
I(\mu) \leq 2 \Big(\lim_{n \to \infty} \frac{\log |b_n|}{\deg \alpha_n} + \esc \Big).
\]
\end{theorem}

\noindent We use a convexity 
property of the logarithmic energy together with Theorem~\ref{T:limiting} to show \eqref{E:lower}.
\medskip

{\noindent{\bf Remarks on $\capa(\nice) \geq 1$ case:}}
As noted earlier,
in \cite{FekeSzego:aeicwr55}, Fekete--Szeg\"{o} showed the converse 
of Result~\ref{R:finite_integers}, i.e., if $\capa(\nice) \geq 1$, then every neighborhood $U$ of $\nice$ 
contains infinitely many 
algebraic integers $\alpha$ so that $\alpha$ and all its conjugates lie in $U$. 
By shrinking $U$, we can get a sequence $\{\alpha_n\}$ of algebraic integers 
such that the sets of conjugates of $\alpha_n$ are eventually contained in any
neighborhood of $\nice$. Bilu \cite{Bilu:limit97} studied limit distributions of the measures $\mu_n$ associated to 
such sequences of algebraic integers, 
more generally, sequences $\{\alpha_n\}$ in $\overline \Q$ satisfying
$h(\alpha_n) \to 0$,
in the case of the unit circle centered at 0.
Using potential-theoretic methods, limit distributions of $\mu_n$ associated to 
sequences $\{\alpha_n\}$ in $\overline \Q$ satisfying
$h_{\nice} (\alpha_n) \to 0$
have been studied for general sets $\nice$ of logarithmic capacity 1 
by Rumely \cite{rumely:bilu99} and Pritsker \cite{Pritsker:crelle11}.
For such sequences $\{\alpha_n\}$ in $\overline \Q$, they showed the 
limiting distribution of $\mu_n$ is given by the equilibrium measure of $\nice$.
Recently potential-theoretic methods
have been employed when $\capa(\nice) >1$ in \cite{Smith:aicpd21} to study limit distributions in connection 
with the Schur--Siegel--Smyth trace problem (see also \cite{OrlSar:ldcai23}).
\medskip

We include the necessary potential-theoretic background in the next section, including a non-standard 
result on logarithmic potentials of certain measures with unbounded support. Sections 3 and 4 give the 
proofs of Theorems~\ref{T:existence} and~\ref{T:limiting}. 
We conclude in Section~\ref{S:remarks} with some further questions.
\medskip

\section{Preliminaries}\label{S:prelim}

\subsection{Potential theory}\label{SS:potential_theory}
We work exclusively in the complex plane $\C$.
The term measure will refer to a positive Borel measure in $\C$ unless otherwise stated. 
Given such a measure $\mu$, we can form its logarithmic potential function
\[
U_\mu (z) := \int \log \frac{1}{\lvert z-w \rvert}\,d\mu(w) 
\]
provided this integral exists. In this case, if $\mu$ has compact support, then 
\begin{equation}\label{potfcn} 
U_\mu (z) =\mu(\C)(-\log |z|+O(1)) \ \hbox{as} \ z\to \infty.
\end{equation}
This is a lower semicontinuous, superharmonic function on $\C$ which is harmonic 
off of the support of $\mu$. The logarithmic energy of $\mu$ is 
\[
I(\mu):= \int \int \log \frac{1}{\lvert z-w \rvert}\,d\mu(w)\,d\mu(z);
\]
this may be $+\infty$. If $\mu$ is a positive measure with unbounded support such that 
$$\int \log^+ |z| \,d\mu(z) < \infty,$$
then $U_\mu$ is still locally integrable and superharmonic, and $I(\mu)>-\infty$ (see \cite{CKL:tpptus98}).
\smallskip

Given a compact set $K\subset \C$, let $\mathcal M(K)$ denote the 
cone of probability measures supported in $K$. The logarithmic capacity of $K$ is 
\[
\capa(K):=\exp[-\inf \{I(\mu): \mu \in \mathcal M(K)\}];
\]
this coincides with the transfinite diameter $d(K)$ in the introduction
(see, for instance, \cite[Ch~2]{Ahlfors:confinv10} or \cite[Ch~5]{ransford:ptCp95}). If $\capa(K)>0$, 
the Frostman conditions (cf., \cite{ransford:ptCp95}) tell us that there is a unique measure 
$\mu_K \in \mathcal M(K)$ with $ \capa(K)=\exp[-I(\mu_K)]$ and $U_{\mu_K}= I(\mu_K)$ 
quasi-everywhere (abbreviated as q.e.) on $K$;
i.e., on all of $K$ except for a polar set. Here, $E\subset \C$ is polar if there exists a subharmonic 
function $u\not \equiv -\infty$ on a neighborhood of $E$ with $E\subset \{u=-\infty\}$; 
for $E$ compact this is equivalent to $\capa(E)=0$. 
If $E \subset \C$ 
is compact and polar, for any $K\subset \C$ compact we have $ \capa(K\cup E)=  \capa(K)$.
The measure $\mu_K$ is called the {\textit{equilibrium measure of}} $K$.
\smallskip

For non-polar compact $K$, we will have 
use of the Green's function 
\[
\gr_K(z,\infty):=-U_{\mu_K}(z)+I(\mu_K)
\]
for (the unbounded component of the complement of) $K$ with pole at $\infty$; note that 
$\gr_K(z,\infty)=0$ q.e. on $K$. This is an upper semicontinuous function which is harmonic 
outside of $K$; from (\ref{potfcn}), it follows that $\gr_{\nice} (z, \infty)- \log^+|z|$
is bounded as stated in the introduction. If $U_{\mu_K}$ (equivalently, $\gr_K$) is continuous,
we say that $K$ is regular. 
This is the same as $U_{\mu_K}=I(\mu_K)$ ($\gr_K=0$) on $K$ (cf. \cite{ransford:ptCp95}).
The Green's function of the circle (or closed disk) of radius $r$ centered at $a$ is 
$\log^+(|z-a|/r)$. For the real interval $K =[-1,1]$, we have
$$\gr_{K} (z, \infty)= \log \vert z+ \sqrt{z^2-1} \vert.$$ 
More generally, the Green's function of $K= [a, b]\subset \R$ is given by
\[\gr_{K} (z, \infty)= \log \vert f(z)+ \sqrt{f(z)^2-1} \vert,\]
where $f(z)= \frac{2}{b-a}z - \frac{a+b}{b-a}$, and $ \capa({K})=\frac{b-a}{4}$.
If $K$ is the Julia set of a polynomial $P$ of degree $n \geq 2$ then its Green's function is
\[
\gr_{K} (z, \infty)= \lim_{j\to \infty} \frac{1}{n^j} \log^+\vert P^j(z)\vert,
\]
where $P^j$ is the $j$-fold composition of $P$ with itself and
$\capa (K) = | {\rm lead}(P) |^{-1/(n-1)}$, where ${\rm lead}(P)$ is the leading coefficient of $P$.
\smallskip

It is a standard fact from potential theory that the logarithmic energy of compactly
supported measures satisfies a convexity 
property (see \cite[Lemma~2.1]{Ahlfors:confinv10} or \cite[Lemma~I.1.8]{SaffTotik:lpwef97}).
It turns out that this property also holds for certain measures that are not 
necessarily compactly supported; later we will need to consider measures $\mu$ with unbounded support. 

\begin{result}[Theorem~2.5 of \cite{CKL:tpptus98}]\label{R:energy_convex}
Let $\tau$ be a signed measure with $\tau(\C)=0$ that satisfies
$\int \log^+|z|\,d\vert \tau\vert < \infty$ and $I(\vert \tau\vert) < \infty$. 
Then the logarithmic energy $I(\tau)$ of $\tau$ is well-defined and $I(\tau) \geq 0$.
In particular,
if $\tau_1$ and $\tau_2$ are probability measures satisfying
$\int \log^+|z|\,d\tau_i < \infty$ and $I(\tau_i) < \infty$ for $i=1,2$, then, considering $\tau=\tau_1-\tau_2$, 
\[
2 \int U_{\tau_1}(z) \,d{\tau_2}(z) \leq {I(\tau_1)+ I(\tau_2)}.
\]
\end{result}
\noindent{In the case where we consider $\tau=\tau_1-\tau_2$, the inequality follows because
\[
I(\tau)={I(\tau_1)+ I(\tau_2)}-2 \int U_{\tau_1}(z) \,d{\tau_2}(z).
\]}

\subsection{Heights and Mahler measure}\label{SS:height}
We recall the notion of height of $\alpha\in \overline \Q$ with respect to a non-polar compact
set $\nice$ as defined in (\ref{height}). Let $P_{\alpha} \in \Z[z]$ be an irreducible 
polynomial with leading coefficient $b_{\alpha}$ and degree $\deg \alpha$.
Then
\[
h_{\nice}(\alpha):= \frac{1}{\deg \alpha} \Big( \log|b_{\alpha}| + \sum_{x \in \Gal (\alpha)} \gr_{\nice} (x, \infty)\Big).
\]

The Mahler measure of a polynomial 
$P_n(z)=a_n\prod_{j=1}^n(z-\alpha_j)$ of degree $n$ (thus $a_n\not = 0$) is 
\[
M(P_n):=\exp\Big[\frac{1}{2\pi}\int_0^{2\pi}\log |P_n(e^{it})|dt\Big]=|a_n|\prod_{j=1}^n \max\{1,|\alpha_j|\},
\]
the latter equality following from Jensen's formula. Since the Green's function for the unit circle 
$\nice=\{|z|=1\}$ is $g_{\nice}(z,\infty)=\log^+|z|:=\max\{0,\log|z|\}$, we have 
\[
\frac{1}{n}\log M(P_n)= \frac{1}{n}\log|a_n|+ \frac{1}{n}\sum_j \gr_{\nice} (\alpha_j, \infty).
\]
This coincides, in the case where $P_n \in \Z[z]$ is an irreducible polynomial with 
roots $\alpha_1,...,\alpha_n\in \overline \Q$, to the Weil height $h(\alpha_j)$.
\smallskip

More generally, following Pritsker \cite{Pritsker:crelle11}, we define the Mahler measure of a polynomial 
$P_n(z)=a_n\prod_{j=1}^n(z-\alpha_j)$ of degree $n$ with respect to any non-polar compact set $\nice$ as 
$$M_{\nice}(P_n):=|a_n|\exp\Big[\sum_{j=1}^n \gr_{\nice}(\alpha_j, \infty)\Big]$$
so that, if $P_n \in \Z[z]$ is an irreducible polynomial with 
roots $\alpha_1,...,\alpha_n\in \overline \Q$, 
we again have the relation 
\begin{equation}\label{htmahler} 
\frac{1}{n}\log M_{\nice}(P_n)=h_{\nice}(\alpha_j).
\end{equation}
Observe that
if $P=P_1P_2 \cdots P_n$ with $\deg P_i>0$ for every $1\leq i \leq n$, then we have
\begin{equation}\label{mahler}
M_{\nice}(P)^{\frac{1}{\deg P}}= \prod_{i=1}^{n}\Big( M_{\nice}(P_i)^{\frac{1}{\deg P_i}}\Big)^{\frac{\deg P_i}{\deg P}}
\end{equation}
where $\sum_{i=1}^{n} \frac{\deg P_i}{\deg P}=1$.
\smallskip

Using this generalized notion of Mahler measure, Pritsker proved the following:

\begin{result}[Pritsker, \cite{Pritsker:crelle11}]\label{pritsker} Let $P_n(z)=a_n\prod_{j=1}^n(z-\alpha_{j,n})$ 
be a sequence of polynomials with integer coefficients and simple zeros. 
Let $\nice \subset \C$ be compact with $\capa(\nice)=1$. Then 
$$\lim_{n\to \infty} (M_{\nice}(P_n))^{1/n}=1$$
if and only if 
\begin{enumerate}
\item $\lim_{n\to \infty} |a_n|^{1/n}=1$;
\item $\lim_{R\to \infty} \lim_{n\to \infty} (\prod_{|\alpha_{j,n}|\geq R} |\alpha_{j,n}|)^{1/n}=1$; and 
\item $\mu_n:=\frac{1}{n} \sum_{j=1}^n \delta_{\alpha_{j,n}}\to \mu_\nice$ in the weak* topology as $n\to \infty$.
\end{enumerate}
\end{result}

As an immediate corollary, we have the following result, which will be used in the proof
of Theorem~\ref{T:existence}.

\begin{lemma}\label{L:cap_1_distri}
Let $\nice \subset \C$ be compact and symmetric about the real axis. Suppose that
$\capa(\nice)= 1$. Let $\{\alpha_n\}$ be a sequence of distinct algebraic integers such that
for every neighborhood $U$ of $\nice$ there exist $N \in \N$ such that
${\Gal (\alpha_n)} \subset U$ for all $n \geq N$. Then
\[
\mu_n:= \frac{1}{\deg \alpha_n} \sum_{x \in \Gal (\alpha_n)} \delta_x 
\longrightarrow  \mu_{\nice} \quad {\text as } \quad n \longrightarrow \infty,
\]
where $\mu_{\nice}$ is the equilibrium measure of $\nice$.
\end{lemma}

Next, we recall that the {\it resultant} of two polynomials 
\begin{align}
P(z)&=a_nz^n+\cdots +a_0=a_n\prod_{j=1}^n (z-\alpha_j) \notag \\
\hbox{and} \quad Q(z)&=b_mz^m +\cdots +b_0=b_m\prod_{k=1}^m (z-\beta_k) \notag
\end{align}
is given by
\[
\hbox{Res}(P,Q):=a_n^mb_m^n \prod_{j=1,...,n; \ k=1,...,m}(\alpha_j-\beta_k).
\]
For $P,Q\in \Z[z]$ it follows that $\hbox{Res}(P,Q)\in \Z$ (cf., \cite{CLO}, page~153). 
Thus, if $P,Q\in \Z[z]$ do not have any factor in common then $\vert \hbox{Res}(P,Q)\vert \geq 1$.
The {\it discriminant} of $P$ is related to the resultant of $P$ and $P'$:
\begin{equation}\label{disc}
\hbox{Disc}(P)=\frac{(-1)^{n(n-1)/2} }{a_n} \cdot \hbox{Res}(P,P')=(-1)^{n(n-1)/2}a_n^{2n-2}
\prod_{j\not=k}(\alpha_j-\alpha_k). 
\end{equation}
The reader will notice the resemblance to the quantity
$\prod_{i<j} \vert \xi_i- \xi_j \vert^{\frac{2}{n(n-1)}}$ in the definition of transfinite diameter from the introduction.
\smallskip

Finally, recall that a sequence of measures $\{\mu_{\beta}\}_{\beta \in B}$ on $\C$ is {\it tight} if for 
any $\epsilon >0$ there exists a compact set $K_{\epsilon}$ such that 
$\mu_{\beta}(\C\setminus K_{\epsilon})< \epsilon$ for all $\beta \in B$. 
If $\{\mu_{\beta}\}_{\beta \in B}$ is tight and uniformly bounded in the total variation 
norm, Prokhorov's theorem implies that it has a convergent subsequence. 
We will need the following application.

\begin{lemma}\label{L:tight}
Let $\{\alpha_n\}$ be a sequence in $ \overline \Q$ such that corresponding 
sequence of heights $\{h_{\nice}(\alpha_n)\}$ is bounded.
If 
\[
\mu_n:= \frac{1}{\deg \alpha_n} \sum_{x \in \Gal (\alpha_n)} \delta_x
\]
is the probability measure uniformly distributed on all conjugates of $\alpha_n$, then 
the sequence $\{\mu_n\}$ has a convergent subsequence in the weak* topology.
\end{lemma}

\begin{proof}
We shall show that the sequence $\{\mu_n\}$ is tight. Then the result follows from Prokhorov's theorem.
Assume that the sequence $\{\mu_n\}$ is not tight. By definition, there exists $\rho >0$ such that
for every compact set $K \subset \C$, we have $\mu_n (\C \setminus K) > \rho$ for some $n \in \N$.
Since $\gr_{\nice}(z, \infty) \to \infty$ as $z \to \infty$, by enlarging $K$, 
we get an $\alpha \in \overline \Q$ from the sequence $\{\alpha_n\}$
for which $h_{\nice}(\alpha)$ is arbitrarily large.
This contradicts the assumption that $\{h_{\nice}(\alpha_n)\}$ is bounded.
Thus the sequence $\{\mu_n\}$ is tight, and there exists a convergent subsequence of $\{\mu_n\}$. 
\end{proof}

\section{Proof of Theorem~\ref{T:existence}}
In this section, we prove Theorem~\ref{T:existence}. We assume that $\nice \subset \C$ is compact and symmetric about the real axis with $0<\capa(\nice) <1$. Since $\nice$ is symmetric about the real axis,
$\gr_{\nice}(z, \infty) =\gr_{\nice} (\bar z, \infty)$; thus level sets of 
$\gr_{\nice} (z, \infty)$ are also symmetric about the real axis. 
For $n \in \N$, define
\[
L_n:=\{ z \in \C : \gr_{\nice} (z, \infty)=n\}.
\]
We choose a compact subset $E_n$ of $L_n$ that is symmetric about the real axis and satisfies $\capa (\nice \cup E_n)=1$.
For sufficiently large $n$, such a set always exists. To see this, we start with a compact subset $E$ of $L_n$
that is invariant under complex conjugation and which satisfies $\capa (\nice \cup E) >1$. For example, if we choose $E$ with $\capa (E) >1$ then $\capa (\nice \cup E) >1$.
Since capacity is a monotone function, by considering a decreasing family of subsets of $L_n$, we eventually get 
a compact subset $E_n$ of $L_n$ that satisfies $\capa (\nice \cup E_n)=1$.
Since $\capa (\nice)<1$, the set $E_n$ is non-polar.
\smallskip

We recall the Fekete-Szeg\"{o} result in \cite{FekeSzego:aeicwr55} discussed in the introduction, i.e., if $\capa(K) \geq 1$ and $K$ is symmetric about the real axis, then every neighborhood $U$ of $K$ 
contains infinitely many 
algebraic integers $\alpha$ so that $\alpha$ and all its conjugates lie in $U$. Let $\nu_n$ be the equilibrium measure of $\nice \cup E_n$. Since $\capa (\nice \cup E_n)=1$,
we can get a sequence $\{\alpha_{n,j}\}$ of algebraic integers such that
${\Gal (\alpha_{n,j})}$ eventually belongs to any neighborhood of $\nice\cup E_n$.
By Lemma~\ref{L:cap_1_distri},
\begin{equation}\label{E:convergence}
\mu_{n,j}:= \frac{1}{\deg \alpha_{n,j}} \sum_{x \in \Gal (\alpha_{n,j})} \delta_x 
\longrightarrow  \nu_{n} \quad {\text as } \quad  j \longrightarrow \infty.
\end{equation}
Fix a neighborhood $U$ of $\nice$. We want to determine what fraction of ${\Gal (\alpha_{n,j})}$ is in $U$ for sufficiently large $j$; thus we want to determine the value of $\nu_n(\nice)$.
To do this, we solve a Dirichlet problem on the domain $D_n$ in the Riemann sphere $\sph$, where $D_n$ is the
component of $ \sph \setminus (\nice \cup E_n)$ containing $\infty$. 
\smallskip

Let $\phi$ be the function on the boundary $\partial D_n$ of $D_n$ such that
$\phi \equiv 0$ on the exterior boundary of $\nice$ and $\phi \equiv 1$ on the exterior boundary of $E_n$.
Since $\phi$ is continous on $\partial D_n$, there exists a unique bounded harmonic function $f$ on $D_n$ such that
$\lim_{z \to \xi} f(z) = \phi(\xi)$ for q.e. $\xi \in \partial D_n$ (cf., Corollary 4.2.6, \cite{ransford:ptCp95}). Then, by definition, $f(\infty)=\nu_n(E_n)$.
We show
\[
f(z) = \frac{\gr_{\nice} (z, \infty) - \gr_{\nice \cup E_n} (z, \infty)}{n}.
\]
Observe that, since $g_{\nice} (z, \infty) \equiv n$ on $E_n$, $\lim_{z \to \xi} f(z) = \phi(\xi)$ for q.e. $\xi \in \partial D_n$.
The function $f$ is harmonic and bounded on $D_n$ with
\begin{align}
f(\infty) &= \lim_{z \to \infty} \frac{\gr_{\nice} (z, \infty) - \gr_{\nice \cup E_n} (z, \infty)}{n} \notag \\
&= - \frac{1}{n} \log \capa(\nice). \notag 
\end{align}
Thus the function $f$ defined above solves the Dirichlet problem for $\phi$.
Therefore, we have $\nu_n(\nice)= 1+ \frac{1}{n} \log \capa(\nice)$. 
\smallskip

Now, owing to the above observations, there exist $j' \in \N$ such that $\alpha_{n,j'}$ $=:\alpha_{n}$ satisfies
\begin{equation}\label{E:portion}
\frac{\sharp \{\Gal (\alpha_n) \cap U\}}{\deg \alpha_n} \geq  1+ \frac{1}{n} \log \capa(\nice) -\frac{1}{n}.
\end{equation}
Note that, since $\capa(\nice) >0$, the right hand side of \eqref{E:portion} converges to 1 as $n \to \infty$.
Since $n$ is arbitrary, we get a sequence $\{\alpha_n\}$ of algebraic integers that satisfy \eqref{E:most}.
\smallskip

We shall now construct a sequence $\{\alpha_n\}$ of distinct algebraic integers
that satisfy \eqref{E:most} for any neighborhood $U$ of $\nice$ and
$h_{\nice}(\alpha_n) \to - \log {\capa(\nice)}$ as $n \to \infty$.
Let $\{\alpha_{n, j}\}$ be a sequence of algebraic integers as above.
There exists $J \in \N$ such that for all $j \geq J$,
we get $\gr_{\nice}(z, \infty) \geq n- \frac{1}{n^2}$ for at least $\deg \alpha_{n, j} \cdot \Big( \nu_n(E_n)-\frac{1}{n^2}\Big)$
many $z \in \Gal (\alpha_{n,j})$.
Therefore, for all $j \geq J$, we get
\begin{equation}\label{E:inq_1}
h_{\nice} (\alpha_{n, j}) \geq \Big(n-\frac{1}{n^2}\Big) \Big(- \frac{1}{n} \log \capa(\nice) - \frac{1}{n^2} \Big).
\end{equation}
Let $r>0$ be such that $\Gal (\alpha_{n,j}) \subset \D(0;r)$ for all $j$. Since $\gr_{\nice}(z, \infty)$ is an upper semicontinuous function, by \eqref{E:convergence}, we get
\begin{equation}\label{E:inq_2}
\limsup_{j \to \infty} h_{\nice} (\alpha_{n, j}) = \limsup_{j \to \infty} \int_{\{\vert z \vert \leq r\}} \gr_{\nice}(z, \infty) \,d\mu_{n,j}(z) \leq 
\int_{\{\vert z \vert \leq r\}} \gr_{\nice}(z, \infty) \,d\nu_n(z).
\end{equation}
Using the definition of $\gr_{\nice}(z, \infty)$ and the fact that $U_{\nu_n} (z)=0$ q.e. on $\nice \cup E_n$, we get
\begin{align}
\int \gr_{\nice}(z, \infty) \,d\nu_n(z) =\int (U_{\mu_{\nice}} (z) - \log \capa(\nice) )\,d\nu_n(z)
&= - \log \capa(\nice)+ \int U_{\nu_n} (z)\,d \mu_{\nice}(z)  \notag \\
&= - \log \capa(\nice). \label{E:inq_3}
\end{align}
By \eqref{E:inq_1}, \eqref{E:inq_2} and \eqref{E:inq_3}, there exist $j' \in \N$ such that $\alpha_{n,j'}$ satisfies
\[
\lvert h_{\nice} (\alpha_{n, j'}) + \log \capa(\nice) \rvert < \frac{1}{n}
\]
and
\[
\frac{\sharp \{\Gal (\alpha_{n,j'}) \cap U_n\}}{\deg \alpha_{n,j'}} \geq 1+ \frac{1}{n} \log \capa(\nice) -\frac{1}{n},
\]
where $U_n:= \{z \in \C : {\rm dist}(x, z) < \frac{1}{n} {\textnormal{ for some }} x \in \nice \}$.
Since $n$ is arbitrary, we get a sequence $\{\alpha_n:= \alpha_{n,j'}\}$ of algebraic integers 
that satisfy the required properties. This finishes the proof of Theorem~\ref{T:existence}.

\begin{remark} \label{rmk} We can show that the sequence of measures $\{\nu_n\}$ constructed above converge 
in the weak* topology to the equilibrium measure $\mu_{\nice}$ (and hence so does the appropriate subsequence 
of the measures $\{\mu_{n,j}\}$). 
First, by construction, $\{\nu_n\}$ is tight and 
uniformly bounded in the total variation norm, so by Lemma \ref{L:tight} there exists a weak* convergent 
subsequence which we continue to denote by $\{\nu_n\}$. Let $\mu$ be its weak* limit.
Writing $\nu_n:=\on_n +\away_n$ where $\on_n:=\nu_n|_{\nice}$ and 
$\away_n:= \nu_n|_{E_n}$, we have $\on_n(\nice)= 1+ \frac{1}{n} \log \capa(\nice)\to 1$ 
so that $\mu$ is the weak* limit of $\{\on_n\}$ and $\mu$ is a probability measure 
supported in $\nice$. Since the logarithmic potential $U_{\mu}$ is lower semicontinuous, we have 
\[
\liminf_{n\to \infty} U_{\on_n}(z_n)\geq U_{\mu}(z) \quad \hbox{for $z\in \C$}
\]
where $z_n\to z$ (cf., Theorem~I.6.8 of \cite{SaffTotik:lpwef97}). 
We will show that $U_{\mu}(z)\leq -\log \capa(\nice)$ for q.e. $z\in \nice$.
For then $I(\mu)\leq -\log \capa(\nice)=I(\mu_{\nice})$ and the result follows.
\smallskip

To this end, since $\nu_n$ is the equilibrium measure of a set of logarithmic capacity equal to 1, we have 
\[
U_{\nu_n}(z)= U_{\on_n}(z)+ U_{\away_n}(z)=0 \quad \hbox{for q.e. $z\in \nice\cup E_n$}.
\]
In particular, the above holds for q.e. $z\in \nice$. Thus
\[
-\limsup_{n\to \infty} U_{\away_n}(z)\geq U_{\mu}(z) \quad \hbox{for q.e. $z\in \nice$}.
\]
Now for $z\in \nice$, since $E_n\to \infty$ and $\away_n(E_n) \to 0$, 
\begin{align}
\limsup_{n\to \infty} U_{\away_n}(z)=\limsup_{n\to \infty} \int_{E_n} \log \frac{1}{|z-w|}\,d\away_n(w) \notag
&= \limsup_{n\to \infty} \int_{E_n} \log \frac{1}{|w|}\,d\away_n(w)\notag \\
&=\limsup_{n\to \infty} \Bigl(-\int_{E_n} \gr_{\nice}(w, \infty)\,d\away_n(w)\Bigr). \notag
\end{align}
Since $\gr_{\nice}(w, \infty) \equiv n$ on $E_n$, we get
\[
\int_{E_n} \gr_{\nice}(w, \infty)\,d\away_n(w)= n \away_n(E_n)
=n \bigl(-\frac{1}{n}\log c(\nice)\bigr) = -\log \capa(\nice).
\]
Therefore, we have
\[
U_{\mu}(z) \leq -\limsup_{n\to \infty} U_{\away_n}(z)=-\log \capa(\nice) \quad \hbox{for q.e. $z\in \nice$}.
\]
\end{remark}
\medskip

\section{Proof of Theorem~\ref{T:limiting}}
Let $\nice \subset \C$ be compact and symmetric about the real axis, and assume that $\nice$
is non-polar. 
Given a sequence $\{\alpha_n\}$ in $\overline\Q$, for simplicity of notation, we let $b_{n}:=b_{\alpha_n}$ and
 $P_n:=P_{\alpha_n}$. 
If $\liminf_{n \to \infty} h_{\nice} (\alpha_n) < \infty$, then there exists a subsequence $\{\alpha_{n_k}\}$ 
for which $\{h_{\nice} (\alpha_{n_k})\}$ and $\{{\log |b_{n_k}|}/{\deg \alpha_{n_k}}\}$ are both convergent.
Thus, without loss of generality,
we make the following assumptions:
\begin{enumerate}[label=$(\alph*)$]
\item\label{I:assumption1} $\deg \alpha_n \to \infty$; 
\item\label{I:assumption2} the sequence of heights $\{h_{\nice} (\alpha_n)\}$ is convergent; and 
\item\label{I:assumption3} the sequence $\{{\log |b_{n}|}/{\deg \alpha_{n}}\}$ is convergent.
\end{enumerate}
As in Lemma~\ref{L:tight}, we let 
$$\mu_n:= \frac{1}{\deg \alpha_n} \sum_{x \in \Gal (\alpha_n)} \delta_x.$$
This lemma ensures that there exists a convergent subsequence of $\{\mu_n\}$. 
In general, the limiting measure $\mu$ 
of the subsequence need not be compactly supported. 
However, in the next result, we shall see that the limiting measure satisfies  
an interesting property that allows us to use potential-theoretic tools.

\begin{result}[Theorem~2.1 of \cite{Pritsker:adsman15}]\label{R:finite_energy}
Let $\{\alpha_n\}$ be a sequence in $\overline \Q$ satisfying \ref{I:assumption1}--\ref{I:assumption3}. 
Let $\{\mu_n\}$ be as above. 
Assume that $\mu_n$ converges to a probability measure $\mu$ in the weak* topology. Then,
$\mu$ has finite logarithmic energy $I(\mu)$. Furthermore, we
have that
\[
\int \log^+ |z| \,d\mu(z) < \infty.
\]
\end{result}

\noindent{This result has been stated in \cite{Pritsker:adsman15} under
the assumption that $\{h(\alpha_n)\}$ is bounded, where $h$ denotes
the standard Weil height. However, as noted in the introduction, there exist a constant $C=C(\nice)$ such that 
$\vert h(\alpha)-h_\nice(\alpha)\vert < C$ for every $\alpha \in \overline \Q$. 
Thus the result holds with our assumptions as well.} 
\smallskip

In fact, Pritsker's result is more general: if $\{P_n:=a_n\prod_{k=1}^n (z-\alpha_{k,n})\}$ is a sequence of 
polynomials in $\Z[z]$ having simple zeros with $\mu_n\to \mu$ and 
$H:=\limsup_{n\to \infty} M(P_n)^{1/n} <\infty$, then 
$I(\mu) \leq 2 \log H$ and 
$\int \log^+ |z| \,d\mu(z)\leq \log H$. The key idea is a uniform upper bound for $I(\mu_R)$ where 
$\mu_R$ is the restriction of the measure $\mu$ to the disk $\bar{\mathbb D}(0;R) :=\{z: |z| \leq R\}$. 
We sketch the main ideas and refer to pp. 130--131 of \cite{Pritsker:adsman15} for details.
Ordering the zeroes 
$
|\alpha_{1,n}|\leq \dots \leq |\alpha_{m_n,n}| \leq R <  |\alpha_{m_n+1,n}| \leq \dots \leq |\alpha_{n,n}|,
$
using the inequality $|\alpha_{j,n}-\alpha_{k,n}|\leq 2|\alpha_{k,n}|$ and \eqref{disc}, one gets 
\[
1\leq |\hbox{Disc}(P_n)| \leq \prod_{1\leq j <k \leq m_n}|\alpha_{j,n}-\alpha_{k,n}|^2 \cdot 4^{(n-1)(n-m_n)}
\Bigl( |a_n|\prod_{m_n <k \leq n}|\alpha_{k,n}|\Bigr)^{2(n-1)}.
\]
Using this estimate, $H:=\limsup_{n\to \infty} M(P_n)^{1/n} <\infty$, and truncating the logarithmic kernel by 
taking $\min \{M,\log \frac{1}{|z-\zeta|}\}$ yields the bound 
\[
I(\mu_R)\leq (1-\mu(\bar{\mathbb D}(0;R)))\log 4 + 2 \log H.
\]
By letting $R \to \infty$, we get $I(\mu) \leq 2 \log H$.
For the finiteness of $\int \log^+ |z| \,d\mu(z)$, one observes that if $\mu_{n,R}$ is the measure 
$\mu_n$ restricted to $\bar{\mathbb D}(0;R)$, then 
\[
\int \log^+ |z|\,d\mu_{n,R}(z)= \frac{1}{n} \log \prod_{k=1}^{m_n}\max\{1, |\alpha_{k,n}|\}
\leq \limsup_{n\to \infty} \log M(P_n)^{1/n}= \log H.
\]
Thus we get $\int \log^+ |z| \,d\mu_{R}(z) \leq \log H$ and the result follows by letting $R \to \infty$.
\smallskip

By Result~\ref{R:finite_energy}, if $\{\mu_n\}$ converges to a probability measure $\mu$, then 
the logarithmic energy of $\mu$ is finite.
We make a definition at this stage. Recall the notation: $\gr_{\nice} (z, \infty)$ denotes the
Green's function of $\nice$ with pole at $\infty$.
\begin{definition}
Let $\{\beta_n\}$ be a sequence in $\overline \Q$.
The {\textit{escape rate}} of the sequence $\{\beta_n\}$ with
respect to $\nice$ is defined as
\begin{equation}\label{E:escape_rate}
{{\mathcal E}(\nice; \{\beta_n\})} = \lim_{r \to \infty} \bigl(\liminf_{n \to \infty}  \frac{1}{\deg \beta_n}\sum_{\substack{x \in \Gal (\beta_n)\\ |x| \geq r}} \gr_{\nice} (x, \infty)\bigr).
\end{equation}
\end{definition}
{\noindent Note that if $\{h_{\nice} (\beta_n)\}$ is bounded, 
the escape rate ${{\mathcal E}(\nice; \{\beta_n\})}$ is a finite non-negative number.}
We now show that the measures that give the limiting distribution of the conjugates
of a sequence in $\overline \Q$ satisfy an interesting arithmetic property. In the case of
a sequence of algebraic integers with bounded conjugates lying near a compact, symmetric set $\Sigma$ with $c(\nice)\geq 1$, 
the arithmetic property
characterizes such measures (see Section~\ref{S:remarks}; also see \cite{Smith:aicpd21}, \cite{OrlSar:ldcai23}).
\smallskip

We now prove Theorem~\ref{T:limiting} in two parts:

\begin{theorem}\label{T:meaure_chara}
Let $\{\alpha_n\}$ be a sequence in $\overline \Q$ satisfying \ref{I:assumption1}--\ref{I:assumption3}. 
If $\{\mu_n\}$ converges to a probability measure $\mu$ in the weak* topology, then
for any polynomial $Q \in \Z[z]$, we have 
\[
\frac{1}{\deg Q} \int \log \vert Q(z) \vert \,d\mu(z) \geq -  \lim_{n \to \infty} \frac{\log |b_n|}{\deg \alpha_n} - \esc,
\]
where $\esc$ is the escape rate of the sequence $\{\alpha_n\}$ as defined in \eqref{E:escape_rate}.
\end{theorem}

\begin{proof}
Fix a polynomial $Q  \in \Z[z]$.
Let $\phi$ be a compactly supported continuous function with $0\leq \phi \leq 1$.
Since $\phi(z) \log \vert Q(z) \vert$ is an upper semicontinuous function that is bounded from above, we have
\begin{align}
\frac{1}{\deg Q} \int \log \vert Q(z) \vert \,d\mu(z) 
\geq& \limsup_{n \to \infty} \frac{1}{\deg Q} \int \phi(z) \log \vert Q(z) \vert \,d\mu_n(z)  \notag \\
\geq& \liminf_{n \to \infty}\frac{1}{\deg Q}\int \log \vert Q(z) \vert \,d \mu_n(z) \notag \\
& -\liminf_{n \to \infty}\frac{1}{\deg Q} \int (1-\phi(z)) \log \vert Q(z) \vert \,d\mu_n(z). \label{E:usc_ineq}
\end{align}
For the last inequality, we use the fact that
$\limsup_{n \to \infty}(k_n+l_n) \geq \liminf_{n \to \infty}k_n+\limsup_{n \to \infty} l_n$. 
Note that, owing to Result~\ref{R:finite_energy}, $\int \log \vert Q(z) \vert \,d\mu(z)$ makes sense.
\smallskip

Since $P_n$ and $Q$ do not have roots in common for large $n$, 
we have $\vert {{\rm Res}} (Q, P_n) \vert \geq 1$. Thus we get
\begin{equation}\label{E:res}
 \frac{1}{\deg Q} \int \log \vert Q(z) \vert \,d\mu_n(z) =  \frac {\log \vert {{\rm Res}} (Q, P_n) \vert - \log |b_n^{\deg Q}| }{\deg Q \cdot \deg P_n} \geq -  \frac{\log |b_n|}{\deg \alpha_n}. 
\end{equation}
Observe that, as $z \to \infty$, $\gr_{\nice} (z, \infty) = \log |z| + O(1)$ and 
$(1/\deg Q)\log \vert Q(z) \vert= \log |z| + O(1)$. 
We denote the disk of radius $r$ centered at the origin by $\D(0;r)$ and its complement by $\D(0;r)^c$.
For every $\delta >0$ there exists
$r>>1$ such that $\mu (\D(0;r)^c) < \delta$ and $\mu_n (\D(0;r)^c) < \delta$ for sufficiently large $n$.
Thus if $0\leq \phi \leq 1$ with $\phi\equiv 1$ on $\D(0;r)$ for sufficiently large $r$, we get
\begin{align}\label{E:green_asym}
\liminf_{n \to \infty} \frac{1}{\deg Q} \int & (1-\phi(z)) \log \vert Q(z) \vert \,d \mu_n(z) \notag\\
&\leq \liminf_{n \to \infty} \frac{1}{\deg \alpha_n}\sum_{\substack{x \in \Gal (\alpha_n)\\ |x| \geq r}}\gr_{\nice}  (x, \infty) + \delta \cdot O(1).
\end{align}
Since the choice of $\phi$ in \eqref{E:usc_ineq} is arbitrary, by using \eqref{E:res} and \eqref{E:green_asym} in \eqref{E:usc_ineq},
it follows that
\[
\frac{1}{\deg Q} \int \log \vert Q(z) \vert \,d\mu(z) \geq -  \lim_{n \to \infty} \frac{\log |b_n|}{\deg \alpha_n} - \esc.
\]
\end{proof}

We now use the above theorem to get an upper bound on $I(\mu)$. Here we continue with the same notation as above.

\begin{theorem}\label{T:energy_upperbound}
Let $\{\alpha_n\}$ be a sequence in $\overline \Q$ satisfying \ref{I:assumption1}--\ref{I:assumption3}. 
If $\{\mu_n\}$ converges to a probability measure $\mu$ in the weak* topology, then
\[
I(\mu) \leq 2 \Big(\lim_{n \to \infty} \frac{\log |b_n|}{\deg \alpha_n} + \esc \Big),
\]
where $\esc$ is the escape rate of the sequence $\{\alpha_n\}$ as defined in \eqref{E:escape_rate}.
\end{theorem}

\begin{proof}
Let $\phi$ be a cut-off function as in Theorem~\ref{T:meaure_chara} so that $0\leq \phi \leq 1$ with 
$\phi\equiv 1$ on $\D(0;r)$ for sufficiently large $r>0$. Since $I(\mu)<\infty$ from Result~\ref{R:finite_energy}, 
given $\delta >0$, we can choose $r>>1$ large so that 
$\int \phi(z) U_{\mu}(z)\,d\mu(z)> I(\mu)-\delta$; in addition, as in the previous lemma, 
for such $r$ we may assume $\mu (\D(0;r)^c) < \delta$ and $\mu_n (\D(0;r)^c) < \delta$. 
Then since $\phi(z) U_{\mu}(z)$ is a lower semicontinuous function that is bounded from below, we have
\begin{align}
I(\mu) -\delta <  \int \phi(z) U_{\mu}(z)\,d\mu(z) \leq& \liminf_{n \to \infty}\int \phi(z) U_{\mu}(z)\,d\mu_n(z) \notag \\ 
\leq& \limsup_{n \to \infty}\int U_{\mu}(z)\,d\mu_n(z) \notag\\
&-\liminf_{n \to \infty} \int (1-\phi(z))U_{\mu}(z)\,d\mu_n(z). \label{E:lsc_ineq}
\end{align}
For the last inequality, we use the fact that
$\liminf_{n \to \infty}(k_n+l_n) \leq \limsup_{n \to \infty}k_n+\limsup_{n \to \infty} l_n$. 
By Fubini's theorem, we see that
\[
 \int U_{\mu}(z)\,d\mu_n(z) =  \int U_{\mu_n}(z)\,d\mu(z)
= - \frac{1}{\deg P_n} \int \log \vert P_n(z) \vert \,d\mu(z)+ \frac{\log |b_n|}{\deg P_n}. 
\]
Since Theorem~\ref{T:meaure_chara} holds for an arbitrary polynomial $Q$ with integer coefficients, in particular it holds
for the polynomials $P_n$. Therefore, fixing $Q=P_n$, for all $m$ sufficiently large, $P_n$ and $P_m$ have no common factors and we get
\begin{equation}\label{E:use_poly_estim}
\limsup_{n \to \infty} \int U_{\mu}(z)\,d\mu_n(z) \leq  \lim_{n \to \infty} \frac{\log |b_n|}{\deg \alpha_n} +
\esc +\lim_{n \to \infty} \frac{\log |b_n|}{\deg P_n}.
\end{equation}
Now as $z\to \infty$,
\begin{align}
U_{\mu}(z)&=\int_{\D(0;r)}\log \frac{1}{|z-w|}\,d\mu(w) + \int_{\D(0;r)^c}\log \frac{1}{|z-w|}\,d\mu(w) \notag\\
&\geq (1-\delta)(- \log |z| + O(1))+ \int_{\D(0;r)^c}\log \frac{1}{|z-w|}\,d\mu(w). \notag 
\end{align}
For any $M>0$ and $|z|>r$, 
$$\int_{\D(0;r)^c}\log \frac{1}{|z-w|}\,d\mu(w)\geq \int_{\D(0;r)^c}\min\Big\{M,\log \frac{1}{|z-w|}\Big\}\,d\mu(w)$$
which is finite; hence, as $z\to \infty$,
$$U_{\mu}(z) \geq - \log |z| + O(1)=-\gr_{\nice}(z, \infty)+O(1) .$$
Therefore, it follows that
\begin{equation}\label{E:poten_asym}
\liminf_{n \to \infty} \int (1-\phi(z)) U_{\mu}(z) \,d \mu_n(z) 
\geq - \liminf_{n \to \infty}\frac{1}{\deg \alpha_n}\sum_{\substack{x \in \Gal (\alpha_n)\\ |x| \geq r}}\gr_{\nice}(x, \infty) 
+ \delta \cdot O(1).
\end{equation}
Since $\delta>0$ is arbitrary, using \eqref{E:use_poly_estim} and \eqref{E:poten_asym} in \eqref{E:lsc_ineq}, it follows that 
\[
I(\mu) \leq 2 \Big(\lim_{n \to \infty} \frac{\log |b_n|}{\deg \alpha_n} + \esc \Big).
\]
\end{proof}

Let $\nice \subset \C$ be compact and symmetric about the real axis with $0< \capa(\nice) <1$.
We now give the proof of the inequality in \eqref{E:lower} using the above theorems. The proof will not require $\capa(\nice) <1$ but if $\capa(\nice) \geq 1$, the inequality is obvious.

\begin{proof}[The proof of Result \ref{R:lower}]
We shall first prove the result when ${\nice}$ is regular (see Section~\ref{S:prelim}).
Let $\{\alpha_n\}$ be a sequence in $\overline \Q$ satisfying \ref{I:assumption1}--\ref{I:assumption3}.
Observe that the result follows if we show that $\lim_{n \to \infty}h_{\nice}(\alpha_n) \geq  -(1/2) \log \capa (\nice)$.
With this notation, we have 
\[
h_{\nice} ({\alpha_n}) = \frac{\log |b_n|}{\deg \alpha_n} + \frac{1}{\deg \alpha_n}  \sum_{x \in \Gal (\alpha_n)}\gr_{\nice} (x, \infty).
\]
Recall $\mu_n$ is the probability measure uniformly distributed on all conjugates
of $\alpha_n$.
By Lemma~\ref{L:tight}, there exists a convergent subsequence of $\{\mu_n\}$. We simply denote this subsequence by $\{\alpha_n\}$ and the corresponding 
sequence of measures by $\{\mu_n\}$. Let $\mu$ be the weak* limit of $\{\mu_n\}$. Note that ${\mu}$ need not be 
compactly supported.
\smallskip

Let $\mu_{\nice}$ denote the equilibrium measure of ${\nice}$. We have
$U_{\mu_{\nice}} (z)=-\gr_{\nice}(z, \infty) -\log\capa (\nice)$ for $z \in \C$. Thus we get
\begin{align}
h_{\nice} ({\alpha_n}) = & \frac{\log |b_n|}{\deg \alpha_n} + \frac{1}{\deg \alpha_n}  \sum_{x \in \Gal (\alpha_n)} \Big(\log \frac{1}{\capa (\nice)} - U_{\mu_{\nice}} (x) \Big) \notag \\
= & \log \frac{1}{\capa (\nice)} + \frac{\log |b_n|}{\deg \alpha_n} - \int U_{\mu_{\nice}} (z) \,d\mu_n(z). \notag
\end{align}
Since $\{h_{\nice}(\alpha_{n})\}$ and $\{{\log |b_{n}|}/{\deg \alpha_{n}}\}$
are both convergent, we get
\begin{align}
\lim_{n \to \infty} h_{\nice} ({\alpha_n}) =  \log \frac{1}{\capa (\nice)} + \lim_{n \to \infty} \frac{\log |b_n|}{\deg \alpha_n} - \lim_{n \to \infty} \int U_{\mu_{\nice}} (z) \,d\mu_n(z) \label{E:height_limit}.
\end{align}
The last limit exists because all other limits in \eqref{E:height_limit} exist. We now estimate the last limit
in \eqref{E:height_limit}.
By our assumption that ${\nice}$ is regular, $U_{\mu_{\nice}}$ is a continuous function.
Let $\phi$ be a non-negative compactly supported continuous function. Thus $\phi(z) U_{\mu_{\nice}} (z)$ is
bounded and continuous, and we get
\begin{align}
\int \phi(z) U_{\mu_{\nice}} \,d\mu(z) &= \lim_{n \to \infty} \int \phi(z) U_{\mu_{\nice}} (z) \,d\mu_n(z) \notag \\
&=\lim_{n \to \infty} \int U_{\mu_{\nice}} (z) \,d\mu_n(z)- \lim_{n \to \infty} \int (1-\phi(z)) U_{\mu_{\nice}} (z) \,d\mu_n(z). \label{E:conti_equa}
\end{align}
Fix $\delta>0$. Since $U_{\mu_{\nice}} (z)= -\log |z| + O(1)$ as $z \to \infty$,
by Result~\ref{R:finite_energy}, there exists $r>0$ such that if $\phi$ is a cut-off function as in Theorem~\ref{T:meaure_chara} so that $0\leq \phi \leq 1$ with $\phi\equiv 1$ on $\D(0;r)$ for sufficiently large $r$, then
\begin{equation}\label{E:delta_ineq}
\int \phi(z) U_{\mu_{\nice}} (z)\,d\mu(z) \leq (1+ \delta) \int U_{\mu_{\nice}} (z) \,d\mu(z).
\end{equation}
As in the proof of Theorem~\ref{T:meaure_chara}, increasing $r$ if necessary,
we have $\mu (\D(0;r)^c) < \delta$ and $\mu_n (\D(0;r)^c) < \delta$ for sufficiently large $n$.
Since $U_{\mu_{\nice}} (z)=-\gr_{\nice}(z, \infty) -\log\capa (\nice)$, for large $r$, as in (\ref{E:poten_asym}), 
we also have
\begin{equation}\label{E:potential_limit}
\lim_{n \to \infty} \int (1-\phi(z)) U_{\mu_{\nice}} (z) \,d\mu_n(z) \leq
- \liminf_{n \to \infty} \frac{1}{\deg \alpha_n}\sum_{\substack{x \in \Gal (\alpha_n)\\ |x| \geq r}} \gr_{\nice} (x, \infty) + \delta \cdot O(1).
\end{equation}
Since $\delta>0$ is arbitrary, \eqref{E:delta_ineq} and \eqref{E:potential_limit} together with \eqref{E:conti_equa} give us
\begin{equation}
\lim_{n \to \infty} \int U_{\mu_{\nice}} (z) \,d\mu_n(z) \leq \int U_{\mu_{\nice}} (z)\,d\mu(z)
- \lim_{r \to \infty}\bigl(\liminf_{n \to \infty} \frac{1}{\deg \alpha_n}\sum_{\substack{x \in \Gal (\alpha_n)\\ |x| \geq r}} \gr_{\nice} (x, \infty)\bigr). \notag
\end{equation}
The limit on the right hand side of the last expression is the escape rate $\esc$ of the sequence $\{\alpha_n\}$
as in Theorem~\ref{T:energy_upperbound}.
By invoking Result~\ref{R:energy_convex} and Theorem~\ref{T:energy_upperbound}, it follows that
\begin{align}
\lim_{n \to \infty} \int U_{\mu_{\nice}} (z) \,d\mu_n(z)
 &\leq \frac{I(\mu_{\nice})+I(\mu)}{2} - \esc \notag\\
 &\leq \frac{1}{2}\log \frac{1}{\capa (\nice)} +  \lim_{n \to \infty} \frac{\log |b_n|}{\deg \alpha_n}. \notag
\end{align}
Using this in \eqref{E:height_limit}, we get
\begin{align}
\lim_{n \to \infty} h_{\nice} ({\alpha_n}) \geq \log \frac{1}{\capa (\nice)} + \lim_{n \to \infty} \frac{\log |b_n|}{\deg \alpha_n} 
-\frac{1}{2}\log \frac{1}{\capa (\nice)} -\lim_{n \to \infty} \frac{\log |b_n|}{\deg \alpha_n} \geq \frac{1}{2}\log \frac{1}{\capa (\nice)}. \notag
\end{align}
Since the sequence $\{\alpha_n\}$ is arbitrary, the result follows 
when $\nice$ is regular.
\smallskip

If $\nice$ is not regular, we consider a sequence $\{\nice_m\}$ of regular compact sets symmetric 
about the real axis such that
$\nice_{m+1} \subset \nice_{m}$ for every $m \in \N$ and $\nice = \cap_m \nice_m$. 
For example, we can take $\nice_m=\{z\in \C: {\rm dist}(z,\nice)\leq 1/m\}$. Note that
$\capa(\nice_{m}) \downarrow \capa(\nice)$. Also, note that
$\gr_{\nice_{m}}(z, \infty) \leq \gr_{\nice}(z, \infty)$ for every $m \in \N$ and $z \in \C$. Therefore,
for any $\alpha\in \overline \Q$, we have $h_{\nice_m} ({\alpha}) \leq h_{\nice} ({\alpha})$
for every $m \in \N$. 
Let $\{\alpha_n\}$ be a sequence in $\overline \Q$ such that $\deg \alpha_n \to \infty$.
Since the sets $\nice_m$ are regular, it follows that
$\liminf_{n \to \infty} h_{\nice} ({\alpha_n}) \geq \liminf_{n \to \infty} h_{\nice_m} ({\alpha_n}) \geq  
-(1/2) \log \capa (\nice_m)$ for every $m \in \N$. Thus
\[
\liminf_{n \to \infty} h_{\nice} ({\alpha_n}) \geq \frac{1}{2} \log \frac{1}{\capa (\nice)}.
\]
\end{proof}

We give an application of the inequality in \eqref{E:lower} to the growth of the leading coefficient of certain
polynomial sequences.
Issai Schur showed 
that restricting the location of zeros of polynomials in $\Z[z]$ with
simple zeros to $[-1,1]$ leads to geometric growth of the leading coefficients of these polynomials.

\begin{result}[Schur \cite{Schur:Uber18}, Satz VII]\label{Schur}
Let $\{Q_n\}$ be a sequence of 
polynomials in $\Z[z]$ with simple zeros such that $\deg Q_n \to \infty$ and all roots of $Q_n$ are in $[-1, 1]$.  
Then 
\[
\liminf_{n \to \infty} \vert {\rm lead}(Q_n) \vert^{1/\deg Q_n} \geq \sqrt{2},
\]
where ${\rm lead}(Q_n)$ is the leading coefficient of $Q_n$.
\end{result}

For $k \in \N$, we denote by $\Z^k[z]$ the set of polynomials in $\Z[z]$ 
whose zeroes are repeated at most $k$ times. 
A version of \eqref{E:lower} is valid in a 
more general setting of a sequence 
of polynomials in $\Z^k[z]$, using the notion of Mahler measure of a polynomial 
in place of the height of an algebraic number.
As an application, we get a lower bound on the growth of the leading coefficients
of a sequence of polynomials in $\Z^k[z]$ whose roots are eventually near $\nice$.

\begin{corollary}\label{C:leading_growth}
Let $\nice \subset \C$ be compact and symmetric about the real axis with $0< \capa(\nice) <1$.
Let $k \in \N$.
Assume that $\{Q_n\}$ is a sequence of polynomials in $\Z^k[z]$ such that $\deg Q_n \to \infty$ 
and
$$\frac{1}{\deg Q_n}\sum_{Q_n(x)=0}\gr_{\nice} (x, \infty)\to 0$$ 
(e.g., if all roots of $\{Q_n\}$ eventually belong to any neighborhood of $\nice$). Then
\[
\liminf_{n \to \infty} \vert {\rm lead}(Q_n) \vert^{1/\deg Q_n} \geq \frac{1}{\sqrt{\capa(\nice)}}.
\]
\end{corollary}

\noindent{For the interval $[-1,1]$ we have $\capa([-1,1])=1/2$; compare with Result \ref{Schur}. 
Pritsker (Proposition~1.7 in \cite{Pritsker:small05}) proved a similar bound with the assumption that the roots lie in the set and with
an additional hypothesis of irreducibility.}
\smallskip

Recall the definition of the Mahler measure
of a polynomial was given in Subsection~\ref{SS:height}; we will also make use of the relation (\ref{mahler}).

\begin{lemma}\label{L:simple_roots}
Let $\nice$ be as in the statement of Corollary~\ref{C:leading_growth}.
Let $k \in \N$.
Let $\{Q_n\}$ be a sequence of polynomials in $\Z^k[z]$ such that $\deg Q_n \to \infty$.
Then we have
\[
\liminf_{n \to \infty}\frac{1}{\deg Q_n}\log M_{\nice} (Q_n) \geq  \frac{1}{2} \log \frac{1}{\capa(\nice)}.
\]
\end{lemma}

\begin{proof}
Fix $\epsilon>0$. As noted in Section~\ref{S:intro},
the Northcott property of $h_{\nice}$ together with \eqref{E:lower} imply that there are only finitely
many $\alpha \in \overline \Q$ such that
$h_{\nice} (\alpha) \leq -(1/2) \log \capa (\nice) - \epsilon$. Thus
there are only finitely many
{\textbf{irreducible}} polynomials $P\in \Z[z]$ satisfying
\[
\frac{1}{\deg P}\log M_{\nice} (P)  \leq \frac{1}{2} \log \frac{1}{\capa(\nice)} - \epsilon.
\]
Let $\mathcal P$ denote the set of such polynomials.
\smallskip

Let $\{Q_n\}$ be a sequence of polynomials in $\Z^k[z]$ such that $\deg Q_n \to \infty$.
Write, for every $n$, $Q_n= Q^1_n Q^2_n$, where
$Q^1_n$ is the product of factors of $Q_n$ that are also in $\mathcal P$ and $Q^2_n$ is the product of
the remaining factors of $Q_n$. Thus, from (\ref{mahler}), 
\begin{equation}\label{E:conve_combi}
\frac{1}{\deg Q_n}\log M_{\nice} (Q_n)= \frac{\deg Q^1_n}{\deg Q_n} \Big(\frac{\log M_{\nice} (Q^1_n)}{\deg Q^1_n}\Big)
+ \frac{\deg Q^2_n}{\deg Q_n}\Big(\frac{\log M_{\nice} (Q^2_n)}{\deg Q^2_n}\Big).
\end{equation}
Similar to \eqref{E:conve_combi},
${\log M_{\nice} (Q^1_n)}/{\deg Q^1_n}$ can be further written as a convex combination
using its irreducible sub-factors. Since these sub-factors are from the set $\mathcal P$, we have
\[
\frac{\log M_{\nice} (Q^1_n)}{\deg Q^1_n} \leq \frac{1}{2} \log \frac{1}{\capa(\nice)} - \epsilon.
\]
Also, ${\log M_{\nice} (Q^2_n)}/{\deg Q^2_n}$ can be further written as a convex combination
using its irreducible sub-factors. Therefore,
\[
\frac{\log M_{\nice} (Q^2_n)}{\deg Q^2_n} > \frac{1}{2} \log \frac{1}{\capa(\nice)} - \epsilon.
\]
Since the zeroes of each $Q_n$ are repeated at most $k$ times, $\{\deg Q^1_n\}$ is bounded.  Therefore, by 
\eqref{E:conve_combi}, there exists $N \in \N$ such that for every $n \geq N$, we have
\[
\frac{1}{\deg Q_n}\log M_{\nice} (Q_n) >  \frac{1}{2} \log \frac{1}{\capa(\nice)} - \epsilon- \frac{\epsilon}{2}.
\]
Since $\epsilon>0$ is arbitrary, the result follows.
\end{proof}

Corollary~\ref{C:leading_growth} follows from Lemma~\ref{L:simple_roots} (if $\nice$ is not regular, 
for the case when all roots of $\{Q_n\}$ eventually belong to any neighborhood of $\nice$, 
we use an argument as in the last part of the proof of \eqref{E:lower}).

\begin{remark}
We can show that \eqref{E:lower} also holds with a slightly different notion of height
(which occurs in the literature; e.g., in \cite{rumely:bilu99}). Given a compact set $\nice$ in $\C$ with $\capa(\nice)>0$,  
let $\Omega_{\nice}$ denote the unbounded component of $\C \setminus \nice$.
We consider the height of $\alpha \in \overline \Q$ with respect to $\nice$ defined as
\[
\widehat h_{\nice} ({\alpha}) =  \frac{1}{\deg\alpha} \Big(\log |b_{\alpha}|+\sum_{x \in \Gal (\alpha) \cap \Omega_{\nice}}\gr_{\nice} (x, \infty)\Big),
\]
where recall $b_{\alpha}$ is the leading coefficient of the irreducible polynomial $P_{\alpha} \in \Z[z]$. 
If $\Gal (\alpha) \cap \Omega_{\nice}=\emptyset$ then we set the sum equal to zero.
Thus $\widehat h_{\nice} ({\alpha}) = 0$ if and only if $\alpha$ is an algebraic integer all of whose conjugates lie in $\C \setminus \Omega_{\nice}$ 
as $\gr_{\nice} (z, \infty) >0$ for $z \in \Omega_{\nice}$.%
\smallskip

Since $\gr_{\nice} (z, \infty) \geq 0$ for $z \in \C$, we have $\widehat h_{\nice} ({\alpha}) \leq h_{\nice} ({\alpha})$.
Furthermore, we have equality for regular sets $\nice$ because for such sets, $\gr_{\nice} (z, \infty)=0$ 
for $z \in \C \setminus \Omega_{\nice}$. Thus \eqref{E:lower} holds when $\nice$ is regular 
with $h_{\nice}$ replaced by $\widehat h_{\nice}$. By
considering a decreasing sequence of regular sets
as in the last part of the proof of \eqref{E:lower}, we see that 
\eqref{E:lower} holds with 
$h_{\nice}$ replaced by $\widehat h_{\nice}$.
Similarly, Lemma~\ref{L:simple_roots} holds with $M_{\nice} $ replaced by $\widehat M_{\nice}$ where 
for $P_n(z)=a_n\prod_{j=1}^n(z-\alpha_j)$, 
\[
\widehat M_{\nice}(P_n):=|a_n|\exp\Big[\sum_{\alpha_j \in \Omega_{\nice}}  \gr_{\nice}(\alpha_j, \infty)\Big].
\]
\end{remark}

 \section{Further discussion}\label{S:remarks}
There remain several relevant questions related to \eqref{E:lower} and Theorem~\ref{T:existence}.
\begin{enumerate}
\item Can the lower bound $-(1/2)\log \capa(\nice)$ in \eqref{E:lower} be improved?
\item Does there exist a sequence of distinct algebraic integers $\{\alpha_n\}$ such that \ref{E:most} holds for every neighborhood
$U$ of $\nice$ and $h_{\nice}(\alpha_n) \to L <- \log {\capa(\nice)}$ as $n \to \infty$?
\item Let $\{\alpha_n\}$ be a  sequence of distinct algebraic integers
such that $\{\alpha_n\}$ satisfy \eqref{E:most} for any neighborhood $U$ of $\nice$ and
$h_{\nice}(\alpha_n) \to - \log {\capa(\nice)}$ as $n \to \infty$. Does the corresponding sequence of 
probability measures $\{\mu_n\}$ converge to the equilibrium measure $\mu_{\nice}$? If not, what other measures $\mu$ can be realized in this fashion?
It follows from \cite[Th\'eor\`eme~2]{FRL:eqpphd06} that if $h_{\nice}(\alpha_n) \to - (1/2)\log {\capa(\nice)}$
then the sequence $\{\mu_n\}$ converges to the equilibrium measure $\mu_{\nice}$.
\end{enumerate}

\noindent Note that, in the case where $\nice$ is a compact set symmetric about the real axis with 
$\capa(\nice)\geq 1$, Orloski and Sardari \cite{OrlSar:ldcai23} have recently characterized the so-called 
{\it arithmetic probability measures} on $\Sigma$, i.e., probability measures $\mu$ on $\nice$ which 
are weak* limits of measures $\{\mu_n\}$ coming from $\{\alpha_n\}$ where all all conjugates 
of $\alpha_n$ eventually lie inside any neighborhood of $\nice$. In our setting when $c(\nice)< 1$, 
necessarily in the second question above we consider the less restrictive hypothesis (\ref{E:most}).


\end{document}